\documentclass[12pt,reqno]{amsart}
\usepackage{amsfonts}

\usepackage{amssymb,amsmath,graphicx,amsfonts,euscript}
\usepackage{color}

\newtheorem{thm}{Theorem}[section]

\newtheorem{rem}[thm]{Remark}

\newtheorem{lemma}[thm]{Lemma}

\allowdisplaybreaks \voffset=-0.2in \numberwithin{equation}{section}

\begin{document}
\bigskip

\centerline{\Large\bf  Remark on the global regularity of 2D MHD equations}
\smallskip

\centerline{\Large\bf    with almost Laplacian magnetic diffusion}

\bigskip

\centerline{Zhuan Ye}

\bigskip

\centerline{Department of Mathematics and Statistics, Jiangsu Normal University, }
\medskip

\centerline{101 Shanghai Road, Xuzhou 221116, Jiangsu, PR China}

\medskip

\centerline{E-mail: \texttt{yezhuan815@126.com
}}

\bigskip
\bigskip
{\bf Abstract:}~~%
Whether or not the classical solutions of the two-dimensional (2D) incompressible
magnetohydrodynamics (MHD) equations with only Laplacian magnetic diffusion (without
velocity dissipation) are globally well-posed is a difficult problem and remains completely open. In this paper, we establish the global regularity of solutions to the 2D incompressible MHD equations with almost Laplacian magnetic diffusion in the whole space. This result can be regarded as a further improvement and generalization of the previous works. Consequently, our result is more closer to the resolution of the global regularity issue
on the 2D MHD equations with standard Laplacian magnetic diffusion.

{\vskip 1mm
 {\bf AMS Subject Classification 2010:}\quad 35Q35; 35B35; 35B65; 76D03.

 {\bf Keywords:}
Generalized MHD equations; Classical solutions; Global regularity.}

\vskip .4in
\section{Introduction}
In this article we focus on the global regularity problem concerning the 2D generalized incompressible magnetohydrodynamic (GMHD) equations of the form in the whole space
\begin{equation}\label{GMHD}
\left\{\aligned
&\partial_{t}u+(u\cdot\nabla)u+\nabla p=(b\cdot\nabla) b,\qquad x \in \mathbb{R}^{2},\, t>0,\\
&\partial_{t}b+(u\cdot\nabla)b+\mathcal{L}b=(b\cdot\nabla)u,
\\
&\nabla\cdot u=0,\ \ \ \nabla\cdot b=0,\\
&u(x,0)=u_{0}(x),\,\,b(x,0)=b_{0}(x),
\endaligned \right.
\end{equation}
where $u=(u_{1}(x,t),\,u_{2}(x,t))$ denotes the velocity, $p=p(x,t)$ the scalar pressure and
$b=(b_{1}(x,t),\,b_{2}(x,t))$ the magnetic field  of the fluid.
$u_{0}(x)$ and $b_{0}(x)$ are the given initial data satisfying $\nabla\cdot u_{0}=\nabla\cdot b_{0}=0$.
The operator $\mathcal{L}$ is a Fourier multiplier with symbol $|\xi|^{2}g(\xi)$, namely,
$$\widehat{\mathcal{L}b}(\xi)=|\xi|^{2}g(\xi)\widehat{b}(\xi),$$
with $g(\xi)=g(|\xi|)$ a radial non-decreasing smooth function satisfying the following two conditions

(a)\,\, $g$ obeys that
$$g(\xi)>0 \quad \mbox{for all}\ \ \xi\neq0;$$

(b)\,\, $g$ is of the Mikhlin-H$\rm\ddot{o}$mander type, namely, a constant $\widetilde{C}>0$ such that
$$|\partial_{\xi}^{k}g(\xi)|\leq \widetilde{C}|\xi|^{-k}|g(\xi)|,\quad k\in \{1,\,2\},\  \forall\,\xi\neq0.$$

\vskip .1in
 The MHD equations model
the complex interaction between the fluid dynamic phenomena, such as
the magnetic reconnection in astrophysics and geomagnetic dynamo in
geophysics, plasmas, liquid metals, and salt water, etc (see, e.g.,
\cite{Davidson01,PF}). The
fundamental concept behind MHD is that magnetic fields can induce currents in a moving conductive fluid, which in turn creates forces on the fluid and also changes the magnetic field itself.
Besides their important physical applications, the MHD equations and GMHD equations are also mathematically significant.
Due to the physical background and mathematical relevance, the MHD equations attracted quite a lot of attention lately from various authors. One of the fundamental problems concerning the MHD equations is whether physically relevant regular solutions remain smooth for all time or they develop finite time singularities.
Actually, there is a considerable body of literature on the global regularity of the  MHD equations with different form of the dissipation (see e.g. \cite{CWYSiam14,CW2011,CRegmiW,FNZ14MM,JZ31114,JZ31115,
TYZ,TYZ113,Wu2003,Wu2011,XuZhang,YB2014JMAA,Y3efg5,YX2014NA,YZhao16} and the references cited therein), and here for our purpose we only recall the notable works about the 2D fractional MHD equations
\begin{equation}\label{2dGMHD}
\left\{\aligned
&\partial_{t}u+(u\cdot\nabla)u+(-\Delta)^{\alpha}u+\nabla p=(b\cdot\nabla) b,\qquad x \in \mathbb{R}^{2},\, t>0,\\
&\partial_{t}b+(u\cdot\nabla)b+(-\Delta)^{\beta}b=(b\cdot\nabla)u,
\\
&\nabla\cdot u=0,\ \ \ \nabla\cdot b=0,\\
&u(x,0)=u_{0}(x),\ \ b(x,0)=b_{0}(x),
\endaligned \right.
\end{equation}
where $(-\Delta)^{\gamma}$ is defined by the Fourier transform, namely
$$
\widehat{(-\Delta)^{\gamma} f}(\xi)=|\xi|^{2\gamma}\hat{f}(\xi).
$$
We remark the convention that by $\alpha=0$ we mean that there is no dissipation in
$(\ref{2dGMHD})_{1}$, and similarly $\beta=0$ represents that there is no diffusion in
$(\ref{2dGMHD})_{2}$.
It is well-known that the 2D MHD
equations with both Laplacian dissipation and magnetic diffusion ($\alpha=\beta=1$) have the global smooth solution (e.g. \cite{ST}). In the completely inviscid case ($\alpha=\beta=0$), the question of
whether smooth solution of the 2D MHD equations develops singularity in finite time appears to be out of reach.
Therefore, it is natural to examine the MHD equations with fractional dissipation (see e.g. \cite{CWYSiam14,FNZ14MM,JZ31114,JZ31115,LinZhang1,Renwxz,TYZ113,YB2014JMAA,Y3efg5,
YX2014NA,
Zhangt} and the references cited therein). Let us review several works about the system (\ref{2dGMHD}) which are very closely related to our study. To the best of our knowledge, the issue of the global regularity for 2D resistive MHD $(\alpha=0,\beta=1)$ is still a challenging open problem (see \cite{CW2011,LZ}) as we are unable to derive the key {\it a priori} estimate of $\|\nabla j\|_{L_{t}^{1}L^{\infty}}$. But if more dissipation is added, then the corresponding system do admit a unique global regular solution. On the one hand, when $\alpha>0,\beta=1$, the global regularity issue was solved by Fan and al. \cite{FNZ14MM}. Very recently, Yuan and Zhao improved this work as they obtained the global regularity of solutions requiring the dissipative operators weaker than any power of the fractional Laplacian.
On the other hand, when $\alpha=0,\beta>1$, the global regularity of smooth solutions for the corresponding system was established in works \cite{CWYSiam14,JZ31115} with quite different approach, which was also improved by Agelas \cite{Agelas} with the $(-\Delta)^{\beta}b$ ($\beta>1$) replaced by $-\Delta\big(\log(e-\Delta)\big)^{\kappa}b$ ($\kappa>1$).

 \vskip .1in
Inspired by the previous works, the aim of this paper is to weaken the operator $\mathcal{L}$ as possible as one can, without losing the global regularity of the system (\ref{GMHD}). More precisely, our main result reads as follows.
\begin{thm}\label{Th1} Let $(u_{0}, b_{0})
\in H^{s}(\mathbb{R}^{2})\times H^{s}(\mathbb{R}^{2})$ for any $s>2$ and satisfy $\nabla\cdot u_{0}=\nabla\cdot b_{0}=0$. Assume that $g(\xi)=g(|\xi|)$ is a radial non-decreasing smooth function satisfying the conditions (a)-(b) and the following growth condition: for any given finite $T\in (0,\,\infty)$ such that
\begin{eqnarray}\label{condi}
\int_{A_{T}}^{\infty}{\frac{dr}{rg(r)}}=C_{T}<\infty,
\end{eqnarray}
where $A_{T}>0$ is the unique solution of the following equation
$$x^{2}g(x)=\frac{1}{T},$$
then the system (\ref{GMHD}) admits a unique global solution such that
$$u,\,\,b\in L^{\infty}\big([0, T]; H^{s}(\mathbb{R}^{2})\big),\ \ \ b\in L^{2}\big([0, T]; H^{s+1}(\mathbb{R}^{2})\big).$$
\end{thm}

\begin{rem}\rm
On the one hand, $A_{T}$ is a decreasing function of $T$ which obviously satisfies
$$\lim_{T\rightarrow \infty}A_{T}=0.$$
On the other hand, $C_{T}$ is a non-decreasing function of $T$ which can satisfy the condition
$$\lim_{T\rightarrow \infty}C_{T}=\infty.$$
\end{rem}

\begin{rem}\rm
The typical examples of $g$ satisfying the conditions (a)-(b) and (\ref{condi}) are
\begin{equation}
\begin{split}
& g(r)=r^{\mu_{1}} \ \ \mbox{with}\ \mu_{1}>0; \\
& g(r)=\big(\ln(1+r)\big)^{\mu_{2}} \ \ \mbox{with}\  \mu_{2}>1;\\
& g(r)=r^{\mu_{3}}\big(\ln(1+r)\big)^{\mu_{4}} \ \ \mbox{with}\ \mu_{3}>0,\, \mu_{4}\geq0;\\
& g(r)=\ln(1+r)\Big(\ln\big(1+\ln(1+r)\big)\Big)^{\mu_{5}} \ \ \mbox{with}\ \mu_{5}>1.\nonumber
\end{split}
\end{equation}
\end{rem}

\begin{rem}\rm
In \cite{Agelas}, Agelas proved the the global regularity of the smooth solution of the system (\ref{GMHD}) with the operator given by $\mathcal{L}=-\Delta\big(\log(e-\Delta)\big)^{\kappa}$ for $\kappa>1$.
Obviously, one can easily check that $\big(\log(e+|\xi|^{2})\big)^{\kappa}$ for $\kappa>1$ satisfies our conditions (a)-(b) and (\ref{condi}).
Clearly, our result generalizes the work \cite{Agelas}, and improves the previous works \cite{CWYSiam14,JZ31115} which require the dissipative $(-\Delta)^{\beta}b$ with $\beta>1$.
\end{rem}
\begin{rem}\rm
For the system (\ref{GMHD}), it remains an open problem whether there
exists a global smooth solution when the function $g(\xi)$ is a positive constant.
\end{rem}

The plan of this paper is as follows: we collect some useful properties of the operator $\mathcal{L}$ in the Section 2, and then we prove Theorem \ref{Th1} in the Section 3. Throughout this paper, the letter $C$ denotes various positive and finite constants whose exact values are
unimportant and may vary from line to line.

\vskip .3in
\section{ The properties of the operator $\mathcal{L}$}\setcounter{equation}{0}
This section is devoted to establishing some properties of the operator $\mathcal{L}$, which are the key components in proving our main theorem.
To begin with, we consider the following linear inhomogeneous equation
\begin{equation}\label{linear}
\left\{\aligned
&\partial_{t}W+\mathcal{L}W=f,
\\
&W(x,0)=W_{0}(x).
\endaligned \right.
\end{equation}
Thanks to the Fourier transform method, the solution of the linear inhomogeneous equation (\ref{linear}) can be explicitly given by
\begin{equation}\label{solve}W(t)=K(t)\ast W_{0}+\int_{0}^{t}{K(t-\tau)\ast f(\tau)\,d\tau},\end{equation}
where the kernel function $K$ satisfies
\begin{eqnarray} 
K(x,t)=\mathcal{F}^{-1}\big(e^{-t|\xi|^{2}g(\xi)}\big)(x).\nonumber
\end{eqnarray}

\vskip .1in
Now we will establish the following estimate which plays an essential role in proving our main theorem.
\begin{lemma}\label{tL31}
For any $k\geq0$ and $s>k-1$, there exists a constant $C$ depending only on $s$ and $k$ such that for any $t>0$
\begin{eqnarray}\label{tok202}
\int_{\mathbb{R}^{2}}{|\xi|^{2s}g(\xi)^{k}e^{-2t|\xi|^{2}g(\xi)}\,d\xi}\leq Ct^{-(s+1)} g(A_{t})^{-(s-k+1)},
\end{eqnarray}
where $A_{t}>0$ is the unique solution of the following equation
\begin{eqnarray}\label{fgetfwf}
x^{2}g(x)=\frac{1}{t}.\end{eqnarray}
\end{lemma}
\begin{proof}[ {Proof of  Lemma \ref{tL31}}]
We split the integral into the following two parts
\begin{eqnarray}
\int_{\mathbb{R}^{2}}{|\xi|^{2s}g(\xi)^{k}e^{-2t|\xi|^{2}g(\xi)}\,d\xi}
=\int_{|\xi|\leq R}{|\xi|^{2s}g(\xi)^{k}e^{-2t|\xi|^{2}g(\xi)}\,d\xi}
+\int_{|\xi|\geq R}{|\xi|^{2s}g(\xi)^{k}e^{-2t|\xi|^{2}g(\xi)}\,d\xi},\nonumber
\end{eqnarray}
where $R>0$ will be fixed hereafter. Recalling $s>k-1$, direct computations yield that the first part admits the following bound
\begin{align}\label{tok203}
\int_{|\xi|\leq R}{|\xi|^{2s}g(\xi)^{k}e^{-2t|\xi|^{2}g(\xi)}\,d\xi} =&\int_{|\xi|\leq R}{|\xi|^{2s}\big(2t|\xi|^{2}\big)^{-k}\big(2t|\xi|^{2}g(\xi)\big)^{k}
e^{-2t|\xi|^{2}g(\xi)}\,d\xi}\nonumber\\ \leq&Ct^{-k}\int_{|\xi|\leq R}{|\xi|^{2s-2k}  \,d\xi}\nonumber\\ \leq&Ct^{-k}\int_{0}^{R}{r^{2s-2k+1}  \,dr}\nonumber\\ \leq&Ct^{-k}R^{2(s-k+1)},
\end{align}
where in the second line we have used the simple fact
\begin{eqnarray}\label{tok204}
\max_{\lambda\geq0}(\lambda^{k}e^{-\lambda})\leq C(k).
\end{eqnarray}
Thanks to the above fact (\ref{tok204}), by some simple computations, we can estimate the remainder term as follows
\begin{align}\label{tok205}
\int_{|\xi|\geq R}{|\xi|^{2s}g(\xi)^{k}e^{-2t|\xi|^{2}g(\xi)}\,d\xi} =&\int_{|\xi|\geq R}{|\xi|^{2s}\big(t|\xi|^{2}\big)^{-k}\big(t|\xi|^{2}g(\xi)\big)^{k}
e^{-t|\xi|^{2}g(\xi)}e^{-t|\xi|^{2}g(\xi)}\,d\xi}\nonumber\\ \leq&Ct^{-k}
\int_{|\xi|\geq R}{|\xi|^{2s-2k}e^{-t|\xi|^{2}g(\xi)}  \,d\xi}
\nonumber\\ \leq&Ct^{-k}
\int_{|\xi|\geq R}{|\xi|^{2s-2k}e^{-t|\xi|^{2}g(R)}  \,d\xi}
\nonumber\\ \leq&Ct^{-k}\int_{R}^{\infty}{r^{2s-2k+1}e^{-tg(R)r^{2}}  \,dr}\nonumber\\ \leq&Ct^{-(s+1)}g(R)^{-(s-k+1)}.
\end{align}
Combining (\ref{tok203}) and (\ref{tok205}), we get
\begin{align}
\int_{\mathbb{R}^{2}}{|\xi|^{2s}g(\xi)^{k}e^{-2t|\xi|^{2}g(\xi)}\,d\xi}
 \leq& Ct^{-k}R^{2(s-k+1)}+Ct^{-(s+1)}g(R)^{-(s-k+1)}
\nonumber\\ \leq& Ct^{-(s+1)}\Big((R^{2}t)^{(s-k+1)}+ g(R)^{-(s-k+1)}\Big).\nonumber
\end{align}
Now taking
$$R^{2}t=\frac{1}{g(R)}\quad \mbox{or}\quad R^{2}g(R)=\frac{1}{t}\quad \mbox{or}\quad R=A_{t},$$
we eventually deduce
\begin{eqnarray}
\int_{\mathbb{R}^{2}}{|\xi|^{2s}g(\xi)^{k}e^{-2t|\xi|^{2}g(\xi)}\,d\xi}
\leq 2Ct^{-(s+1)} g(R)^{-(s-k+1)}=2Ct^{-(s+1)} g(A_{t})^{-(s-k+1)}
.\nonumber
\end{eqnarray}
Thus, we get the desired result.
This concludes the proof of Lemma \ref{tL31}.
\end{proof}

\vskip .1in
With the help of Lemma \ref{tL31}, one can conclude the following results.
\begin{lemma}\label{L32}
For any $s\geq0$, there exists a constant $C$ depending only on $s$ such that for any $t>0$
\begin{eqnarray}\label{tok206}
\|K(t)\|_{\dot{H}^{s}}\leq Ct^{-\frac{s+1}{2}}g(A_{t})^{-\frac{s+1}{2}}.
\end{eqnarray}
In particular, it holds
\begin{eqnarray}\label{tok207}
\|K(t)\|_{L^{\infty}}\leq Ct^{-1}g(A_{t})
^{-1},
\end{eqnarray}
where $A_{t}$ is given by (\ref{fgetfwf}).
\end{lemma}
\begin{proof}[ {Proof of Lemma \ref{L32}}]
According to (\ref{tok202}), one immediately obtains
\begin{align}
\|K(t)\|_{\dot{H}^{s}}^{2} =&
\int_{\mathbb{R}^{2}}{|\xi|^{2s}|\widehat{K}(\xi,t)|^{2}} \,d\xi\nonumber\\
 =&
\int_{\mathbb{R}^{2}}{|\xi|^{2s}e^{-2t|\xi|^{2}g(\xi)}} \,d\xi
\nonumber\\ \leq& Ct^{-(s+1)}g(A_{t})
^{-(s+1)}.\nonumber
\end{align}
By the simple interpolation inequality and (\ref{tok206}), we have
\begin{align}
\|K(t)\|_{L^{\infty}} \leq& C \|K(t)\|_{L^{2}}^{\frac{1}{2}}\|K(t)\|_{\dot{H}^{2}}^{\frac{1}{2}}
\nonumber\\ \leq& Ct^{-\frac{1}{4}}g(A_{t})
^{-\frac{1}{4}}t^{-\frac{3}{4}}g(A_{t})
^{-\frac{3}{4}}
\nonumber\\
 =&
Ct^{-1}g(A_{t})
^{-1}.\nonumber
\end{align}
Therefore, we finally conclude the proof of Lemma \ref{L32}.
\end{proof}

\vskip .1in
Next we would like to show the following lemma.
\begin{lemma}\label{tL33}
 There exists a constant $C$ such that for any $t>0$
\begin{eqnarray}\label{tok208}
\|\nabla^{2}K(t)\|_{L^{1}}\leq Ct^{-1}g(A_{t})
^{-1},
\end{eqnarray}
where $A_{t}$ is given by (\ref{fgetfwf}).
\end{lemma}
\begin{proof}[{Proof of Lemma \ref{tL33}}]
Since for any radial function $g$ in $\mathbb{R}^{n}$, one may check that
$$\widehat{g}(x)=\check{g}(x),$$
where $\check{g}$ denotes the inverse Fourier transform of $g$.
Since $\widehat{K}(.,t)$ is a radial function, then we infer that ${K}(.,t)$ is also a radial function. Based on this observation and the following interpolation inequality (see the end of this lemma)
\begin{eqnarray}\label{tdsfsdgervv}\|\widehat{h}\|_{L^{1}}\leq C \|h\|_{L^{2}}^{\frac{1}{2}}\|h\|_{\dot{H}^{2}}^{\frac{1}{2}}\end{eqnarray}
it yields that
\begin{align}\label{tok209}
\|\nabla^{2}K(t)\|_{L^{1}}&= \|\mathcal{F}(\xi^{2}\check{K}(\xi,t))\|_{L^{1}}
\nonumber\\&= \|\mathcal{F}(\xi^{2}\widehat{K}(\xi,t))\|_{L^{1}}
\nonumber\\&\leq C \|\xi^{2}\widehat{K}(\xi,t)\|_{L^{2}}^{\frac{1}{2}}
\|\nabla_{\xi}^{2}(\xi^{2}\widehat{K}(\xi,t))\|_{L^{2}}^{\frac{1}{2}},
\end{align}
where we have denoted by $\xi^{2}$ the matrix $\xi\otimes \xi$.
By (\ref{tok206}), one gets
\begin{eqnarray}\label{tok210}
\|\xi^{2}\widehat{K}(\xi,t)\|_{L^{2}}
\leq Ct^{-\frac{3}{2}}g(A_{t})
^{-\frac{3}{2}}.
\end{eqnarray}
It follows from some direct computations that
$$|\nabla_{\xi}^{2}(\xi^{2}\widehat{K}(\xi,t))|\leq C(|\widehat{K}(\xi,t)|+|\xi|\,|\nabla_{\xi}\widehat{K}(\xi,t)|
+|\xi|^{2}|\nabla_{\xi}^{2}\widehat{K}(\xi,t)|).$$
Recalling $\widehat{K}(\xi,t)=e^{-t|\xi|^{2}g(\xi)}$ and the property $|\partial_{\xi}^{k}g(\xi)|\leq \widetilde{C}|\xi|^{-k}|g(\xi)|,\, k\in \{1,\,2\}$, it is not hard to check
\begin{align}
|\nabla_{\xi}\widehat{K}(\xi,t)|
&\leq  C(t|\xi|\,|g(\xi)|+t|\xi|^{2}|\partial_{\xi}g(\xi)|)|\widehat{K}(\xi,t)|\nonumber\\
&\leq   Ct|\xi|\,|g(\xi)|\,|\widehat{K}(\xi,t)|\nonumber
\end{align}
and
\begin{align}\label{oklll4788}
|\nabla_{\xi}^{2}\widehat{K}(\xi,t)|
\leq & C(t^{2}|\xi|^{2}|g(\xi)|^{2}+t^{2}|\xi|^{4}|\partial_{\xi}g(\xi)|^{2}+t|g(\xi)|+
t|\xi||\partial_{\xi}g(\xi)|+t|\xi|^{2}|\partial_{\xi}^{2}g(\xi)|)
\nonumber\\& \times|\widehat{K}(\xi,t)|\nonumber\\
\leq &C(t^{2}|\xi|^{2}|g(\xi)|^{2}+t|g(\xi)|)|\widehat{K}(\xi,t)|.
\end{align}
Therefore, we obtain
$$|\nabla_{\xi}^{2}(\xi^{2}\widehat{K}(\xi,t))|\leq C(1+t|\xi|^{2}|g(\xi)|+t^{2}|\xi|^{4}|g(\xi)|^{2})|\widehat{K}(\xi,t)|.$$
We appeal to (\ref{tok202}) to get
\begin{align}\label{tok211}
\|\nabla_{\xi}^{2}(\xi^{2}\widehat{K}(\xi,t))\|_{L^{2}}^{2}&\leq
C\int_{\mathbb{R}^{2}}{e^{-2t|\xi|^{2}g(\xi)}\,d\xi}+C
t^{2}\int_{\mathbb{R}^{2}}{|\xi|^{4}g(\xi)^{2}e^{-2t|\xi|^{2}g(\xi)}\,d\xi}
\nonumber\\& \quad+C
t^{4}\int_{\mathbb{R}^{2}}{|\xi|^{8}g(\xi)^{4}e^{-2t|\xi|^{2}g(\xi)}\,d\xi}
\nonumber\\
&\leq
Ct^{-1}g(A_{t})^{-1}+
Ct^{2}t^{-3}g(A_{t})^{-1}
+C
t^{4}t^{-5}g(A_{t})^{-1}\nonumber\\
&\leq
Ct^{-1}g(A_{t})^{-1}.
\end{align}
Inserting (\ref{tok210}) and (\ref{tok211}) into (\ref{tok209}) yields
$$\|\nabla^{2}K(t)\|_{L^{1}}\leq Ct^{-1}g(A_{t})
^{-1},$$
which is nothing but the desired estimate (\ref{tok208}). Finally, let us show (\ref{tdsfsdgervv})
\begin{align}
\|\widehat{h}\|_{L^{1}}&= \int_{\mathbb{R}^{2}} |\widehat{h}(\xi)|\,d\xi\nonumber\\
&= \int_{|\xi|\leq N} |\widehat{h}(\xi)|\,d\xi+\int_{|\xi|\geq N} |\widehat{h}(\xi)|\,d\xi
\nonumber\\
&= \int_{|\xi|\leq N} |\widehat{h}(\xi)|\,d\xi+\int_{|\xi|\geq N}|\xi|^{-2}|\xi|^{2} |\widehat{h}(\xi)|\,d\xi
\nonumber\\
&\leq C N\Big(\int_{|\xi|\leq N} |\widehat{h}(\xi)|^{2}\,d\xi\Big)^{\frac{1}{2}}+CN^{-1}\Big(\int_{|\xi|\geq N}|\xi|^{4} |\widehat{h}(\xi)|^{2}\,d\xi\Big)^{\frac{1}{2}}
\nonumber\\
&\leq C N\|h\|_{L^{2}}+CN^{-1}\|h\|_{\dot{H}^{2}}
\nonumber\\
&\leq C
\|h\|_{L^{2}}^{\frac{1}{2}}\|h\|_{\dot{H}^{2}}^{\frac{1}{2}},\nonumber
 \end{align}
where in the last line we have chosen $N$ as
$$N=\Big(\frac{\|h\|_{\dot{H}^{2}}}{\|h\|_{L^{2}}}\Big)^{\frac{1}{2}}.$$
We thus complete the proof of the lemma.
\end{proof}

\vskip .1in

Now we are ready to show the following key lemma.
\begin{lemma}\label{khrdlyvhbg}
Under the assumptions of Theorem \ref{Th1}, there holds
\begin{align}
\int_{0}^{T}{
\|\nabla K(t)\|_{L^{2}}\,dt}+\int_{0}^{T}{
\| K(t)\|_{L^{2}}^{2}\,dt}+\int_{0}^{T}{
\|K(t)\|_{L^{\infty}}\,dt}+\int_{0}^{T}{
\|\nabla^{2} K(t)\|_{L^{1}}\,dt}\leq C(T).\nonumber
\end{align}
\end{lemma}

\begin{proof}[{Proof of Lemma \ref{khrdlyvhbg}}]
Making use of Lemma \ref{L32} and Lemma \ref{L33}, it follows that
$$
\|\nabla K(t)\|_{L^{2}}+
\| K(t)\|_{L^{2}}^{2}+
\|K(t)\|_{L^{\infty}}+
\|\nabla^{2} K(t)\|_{L^{1}}\leq Ct^{-1}g(R)^{-1},$$
where $R$ satisfies
$$
R^{2}g(R)=\frac{1}{t}.
$$
Taking $R$-derivative of the both sides of the above equation, it is easy to check
$$2Rg(R)+R^{2}g'(R)=-\frac{1}{t^{2}}\frac{dt}{dR}$$
or equivalently
$$dt=-t^{2}\big(2Rg(R)+R^{2}g'(R)\big)dR=-R^{-4}g(R)^{-2}\big(2Rg(R)+R^{2}g'(R)\big)dR.$$
By direct computations, we have that
\begin{align}
\int_{0}^{T}{t^{-1}g(R)^{-1}\,dt}&= -\int_{+\infty}^{A_{T}}{R^{-2}g(R)^{-2}\big(2Rg(R)+R^{2}g'(R)\big)\, dR}\nonumber\\&= \int_{A_{T}}^{+\infty}{R^{-2}g(R)^{-2}\big(2Rg(R)+R^{2}g'(R)\big)\, dR}
\nonumber\\&= 2\int_{A_{T}}^{+\infty}{R^{-1}g(R)^{-1}\, dR}+ \int_{A_{T}}^{+\infty}{g(R)^{-2}g'(R) \, dR}
\nonumber\\&= 2\int_{A_{T}}^{+\infty}{R^{-1}g(R)^{-1}\, dR}+g(A_{T})^{-1}-g(+\infty)^{-1}
\nonumber\\&\leq 2\int_{A_{T}}^{+\infty}{R^{-1}g(R)^{-1}\, dR}+TA_{T}^{2}\nonumber\\&\leq  C(T),\nonumber
\end{align}
where we have applied (\ref{condi}) and the fact $A_{T}^{2}g(A_{T})=\frac{1}{T}$ or $g(A_{T})^{-1}=TA_{T}^{2}$. Consequently, it is obvious to check that
$$
\int_{0}^{T}{
\|\nabla K(t)\|_{L^{2}}\,dt}+\int_{0}^{T}{
\| K(t)\|_{L^{2}}^{2}\,dt}+\int_{0}^{T}{
\|K(t)\|_{L^{\infty}}\,dt}+\int_{0}^{T}{
\|\nabla^{2} K(t)\|_{L^{1}}\,dt}\leq C(T).\nonumber
$$
This concludes the proof of the lemma.
\end{proof}

\vskip .2in
\section{ The proof of Theorem \ref{Th1}}\setcounter{equation}{0}
This section is devoted to the proof of Theorem \ref{Th1}, the global existence and unique-
ness of smooth solution to the system $(\ref{GMHD})$.
Before the proof, we will state a notation.
For a quasi-Banach space $X$ and for any $0<T\leq\infty$, we use standard notation $L^{p}(0,T;X)$ or $L_{T}^{p}(X)$ for
the quasi-Banach space of Bochner measurable functions $f$ from $(0,T)$ to $X$ endowed with the norm
\begin{equation*}
\|f\|_{L_{T}^{p}(X)}:=\left\{\aligned
&\left(\int_{0}^{T}{\|f(.,t)\|_{X}^{p}\,dt}\right)^{\frac{1}{p}}, \,\,\,\,\,1\leq p<\infty,\\
&\sup_{0\leq t\leq T}\|f(.,t)\|_{X},\qquad\qquad p=\infty.
\endaligned\right.
\end{equation*}

\vskip .2in
By the classical hyperbolic method, there exists a finite time $T_{0}$ such that the system (\ref{GMHD}) is local well-posedness in the interval $[0,\,T_{0}]$ in $H^{s}$ with $s>2$.
Thus, it is sufficient to establish {\it a priori} estimates in the interval $[0,\,T]$ for the given $T>T_{0}$.

\vskip .1in

We first state the basic $L^{2}$-estimate of the system (\ref{GMHD}).
\begin{lemma}\label{L31}
Let $(u_{0},b_{0})$ satisfy the conditions stated in Theorem \ref{Th1}, then it holds
\begin{eqnarray}\label{t301}
\|u(t)\|_{L^{2}}^{2}+\|b(t)\|_{L^{2}}^{2}+ 2\int_{0}^{t}{
\|\mathcal{L}^{\frac{1}{2}}b(\tau)\|_{L^{2}}^{2}\,d\tau}=\|u_{0}\|_{L^{2}}^{2}
+\|b_{0}\|_{L^{2}}^{2},
\end{eqnarray}
where $\mathcal{L}^{\frac{1}{2}}$ is defined by
$$\widehat{\mathcal{L}^{\frac{1}{2}}b}(\xi)=|\xi| \sqrt{g(\xi)}\,\widehat{b}(\xi).$$
\end{lemma}
\begin{proof}[{Proof of  Lemma \ref{L31}}]
Taking the inner products of $(\ref{GMHD})_{1}$ with $u$ and
$(\ref{GMHD})_{2}$ with $b$, adding
the results and invoking the following cancelation identity
$$\int_{\mathbb{R}^{2}} {(b \cdot \nabla)b\cdot u
\,dx}+\int_{\mathbb{R}^{2}}  { (b \cdot \nabla)u\cdot b\,dx}=0,$$
we obtain that
\begin{eqnarray}
\frac{1}{2}\frac{d}{dt}(\|u(t)\|_{L^{2}}^{2}+\|b(t)\|_{L^{2}}^{2})+\|\mathcal{L}^{\frac{1}{2}}b
\|_{L^{2}}^{2}=0,\nonumber
\end{eqnarray}
where we have used the following estimate due to the Plancherel theorem
\begin{eqnarray}
\int_{\mathbb{R}^{2}}{\mathcal{L}b\,b\,dx}=\int_{\mathbb{R}^{2}}{\widehat{\mathcal{L}b}(\xi)\,\widehat{b}(\xi)
\,d\xi}=\int_{\mathbb{R}^{2}}{|\xi|^{2} {g(\xi)}|\widehat{b}(\xi)|^{2}
\,d\xi}=\int_{\mathbb{R}^{2}}{|\widehat{\mathcal{L}^{\frac{1}{2}}b}(\xi)|^{2}\,d\xi}
=\|\mathcal{L}^{\frac{1}{2}}b\|_{L^{2}}^{2}.\nonumber
\end{eqnarray}
Integrating the above inequality implies (\ref{t301}). This ends the proof of  Lemma \ref{L31}.
 \end{proof}

\vskip .1in
In order to get the $H^{1}$ estimate on $(u, b)$, we apply
$\nabla\times$ to the MHD equations (\ref{GMHD}) to obtain the
governing equations for the vorticity $\omega:=\nabla\times u=\partial_{x_{1}}u_{2}-\partial_{x_{2}}u_{1}$ and the current $j:=\nabla\times b=\partial_{x_{1}}b_{2}-\partial_{x_{2}}b_{1}$ as
follows
\begin{equation}\label{VMHD}
\left\{\aligned
&\partial_{t}\omega+(u\cdot\nabla)\omega=(b\cdot\nabla) j,\\
&\partial_{t}j+(u\cdot\nabla)j+\mathcal{L}j=(b\cdot\nabla)\omega+T(\nabla u, \nabla b),
\endaligned \right.
\end{equation}
where $T(\nabla u, \nabla
b)=2\partial_{x_{1}}b_{1}(\partial_{x_{2}}u_{1}+\partial_{x_{1}}u_{2})
-2\partial_{x_{1}}u_{1}(\partial_{x_{2}}b_{1}+\partial_{x_{1}}b_{2}).$

\vskip .1in
We now prove that any classical solution of (\ref{GMHD}) admits a global
$H^1$-bound, as stated in the following lemma.
\begin{lemma}\label{tL32}
Let $(u_{0},b_{0})$ satisfy the conditions stated in Theorem \ref{Th1}, then it holds for any $t\in[0,\,T]$
\begin{eqnarray}\label{t303}
\|\omega(t)\|_{L^{2}}^{2}+\|j(t)\|_{L^{2}}^{2}+ \int_{0}^{t}{
\|\mathcal{L}^{\frac{1}{2}}j(\tau)\|_{L^{2}}^{2}\,d\tau}\leq C(T,u_{0},b_{0}).
\end{eqnarray}
\end{lemma}
\begin{proof}[{Proof of  Lemma \ref{tL32}}]
Taking the inner products of $(\ref{VMHD})_{1}$ with $\omega$, $(\ref{VMHD})_{2}$ with $j$, adding them up and using the incompressible condition as well as the following fact
$$\int_{\mathbb{R}^{2}} {(b \cdot \nabla j) \omega
\,dx}+\int_{\mathbb{R}^{2}}  { (b \cdot \nabla\omega) j\,dx}=0,$$
it yields
\begin{eqnarray}
\frac{1}{2}\frac{d}{dt}(\|\omega(t)\|_{L^{2}}^{2}+\|j(t)\|_{L^{2}}^{2})
+\|\mathcal{L}^{\frac{1}{2}}j\|_{L^{2}}^{2}=\int_{\mathbb{R}^{2}} {T(\nabla u, \nabla
b)j
\,dx}.\nonumber
\end{eqnarray}
We obtain by direct computations that
\begin{align}
\|\nabla \phi\|_{L^{2}}^{2}
&= C\|\xi \widehat{\phi}(\xi)\|_{L^{2}}^{2}\nonumber\\
&= C\int_{\mathbb{R}^{2}}|\xi|^{2}|\widehat{\phi}(\xi)|^{2}\,d\xi
\nonumber\\
&= C\int_{|\xi|\leq 1}|\xi|^{2}|\widehat{\phi}(\xi)|^{2}\,d\xi+C\int_{|\xi|\geq 1}|\xi|^{2}|\widehat{\phi}(\xi)|^{2}\,d\xi
\nonumber\\
&= C\int_{|\xi|\leq 1}|\widehat{\phi}(\xi)|^{2}\,d\xi+C\int_{|\xi|\geq 1}\frac{1}{g(\xi)}|\xi|^{2}g(\xi)|\widehat{\phi}(\xi)|^{2}\,d\xi
\nonumber\\
&\leq C\int_{\mathbb{R}^{2}}|\widehat{\phi}(\xi)|^{2}\,d\xi+C\int_{|\xi|\geq 1}\frac{1}{g(1)}|\xi|^{2}g(\xi)|\widehat{\phi}(\xi)|^{2}\,d\xi
\nonumber\\
&\leq C_{1}\|\phi\|_{L^{2}}^{2}+C_{2}\|\mathcal{L}^{\frac{1}{2}}\phi\|_{L^{2}}^{2},
\nonumber
\end{align}
which implies
\begin{eqnarray}\label{adduse}
\|\nabla \phi\|_{L^{2}}\leq C_{1}\|\phi\|_{L^{2}}+C_{2}\|\mathcal{L}^{\frac{1}{2}}\phi\|_{L^{2}}.
\end{eqnarray}
Therefore, we get by virtue of the easy interpolation inequality that
\begin{align}
\int_{\mathbb{R}^{2}} {T(\nabla u, \nabla
b)j
\,dx}&\leq C\|\nabla u\|_{L^{2}}\|\nabla b\|_{L^{4}}\|j\|_{L^{4}}\nonumber\\
&\leq C\|\omega\|_{L^{2}}\|j\|_{L^{4}}^{2}\nonumber\\
&\leq C\|\omega\|_{L^{2}}\|j\|_{L^{2}}\|\nabla j\|_{L^{2}}\nonumber\\
&\leq C\|\omega\|_{L^{2}}\|\nabla b\|_{L^{2}}\|\nabla j\|_{L^{2}}\nonumber\\
&\leq C\|\omega\|_{L^{2}}(\|b\|_{L^{2}}+ \|\mathcal{L}^{\frac{1}{2}}b\|_{L^{2}})
(\|j\|_{L^{2}}+ \|\mathcal{L}^{\frac{1}{2}}j\|_{L^{2}})\nonumber\\
&\leq \frac{1}{2}\|\mathcal{L}^{\frac{1}{2}}j\|_{L^{2}}^{2}
+C(1+\|b\|_{L^{2}}^{2}+\|\mathcal{L}^{\frac{1}{2}}b\|_{L^{2}}^{2})(\|\omega\|_{L^{2}}^{2}+\|j\|_{L^{2}}^{2}).\nonumber
\end{align}
As a result, one has
$$\frac{d}{dt}(\|\omega(t)\|_{L^{2}}^{2}+\|j(t)\|_{L^{2}}^{2})
+\|\mathcal{L}^{\frac{1}{2}}j\|_{L^{2}}^{2}\leq C(1+\|b\|_{L^{2}}^{2}+\|\mathcal{L}^{\frac{1}{2}}b\|_{L^{2}}^{2})(\|\omega\|_{L^{2}}^{2}+\|j\|_{L^{2}}^{2}).$$
From the classical Gronwall lemma and (\ref{t301}), we easily get (\ref{t303}). The proof of the lemma is now achieved.
\end{proof}
\vskip .1in

Making use of the global bounds obtained
in Lemma \ref{tL32}, we obtain a global bound for the $L_{t}^{\infty}L^{\infty}$-norm of $b$ as follows.
\begin{lemma}\label{L33}
Let $(u_{0},b_{0})$ satisfy the conditions stated in Theorem \ref{Th1}, then it holds for any $t\in[0,\,T]$
\begin{eqnarray}\label{t304}
\|b\|_{L_{t}^{\infty}L^{\infty}}\leq C(T,u_{0},b_{0}).
\end{eqnarray}
\end{lemma}
\begin{proof}[{Proof of  Lemma \ref{L33}}]
we rewrite the second equation of (\ref{GMHD}) as
$$\partial_{t}b+\mathcal{L}b=\nabla\cdot(b\otimes u)-\nabla\cdot
(u\otimes b).$$
By (\ref{solve}), we can check that
$$b(t)=K(t)\ast b_{0}+\int_{0}^{t}{K(t-\tau)\ast [\nabla\cdot(b\otimes u)-\nabla\cdot
(u\otimes b)](\tau)\,d\tau}.$$
Taking $L^{\infty}$-norm in terms of space variable and using the Young inequality, it implies
\begin{align}\|b(t)\|_{L^{\infty}}&\leq \|K(t)\ast b_{0}\|_{L^{\infty}}+\int_{0}^{t}{\|\nabla K(t-\tau)\ast [(b\otimes u)-
(u\otimes b)](\tau)\|_{L^{\infty}}\,d\tau}\nonumber\\&\leq \|\widehat{K(t)\ast b_{0}}\|_{L^{1}}+C\int_{0}^{t}{\|\nabla K(t-\tau)\|_{L^{2}}\|[(b\otimes u)-
(u\otimes b)](\tau)\|_{L^{2}}\,d\tau}\nonumber\\&\leq \|\widehat{K(t)} \widehat{b_{0}}\|_{L^{1}}+C\int_{0}^{t}{\|\nabla K(t-\tau)\|_{L^{2}}\|u(\tau)\|_{L^{4}}\|b(\tau)\|_{L^{4}}\,d\tau}
\nonumber\\&\leq  \|\widehat{b_{0}}\|_{L^{1}}+C\|\nabla K(\tau)\|_{L_{t}^{1}L^{2}}\|u(\tau)\|_{L_{t}^{\infty}L^{4}}
\|b(\tau)\|_{L_{t}^{\infty}L^{4}}
\nonumber\\&\leq C \|b_{0}\|_{L^{2}}^{\frac{1}{2}}\|b_{0}\|_{\dot{H}^{2}}^{\frac{1}{2}}+C(t,u_{0},b_{0})\nonumber\\&\leq  C(t,u_{0},b_{0}),\nonumber
\end{align}
where we have used \eqref{tdsfsdgervv} and the following fact
\begin{eqnarray}\label{t305}
 \int_{0}^{t}{
\|\nabla K(\tau)\|_{L^{2}}\,d\tau}\leq C(t).
\end{eqnarray}
This allows us to derive  
\begin{eqnarray} 
\|b\|_{L_{t}^{\infty}L^{\infty}}\leq C(t,u_{0},b_{0}),\nonumber
\end{eqnarray}
which is the desired bound \eqref{t304}.
Consequently, the proof of  Lemma \ref{L33} is completed.
\end{proof}
\vskip .1in

With the help of the global bounds obtained
in the above lemmas, we obtain a global bound for the $L_{t}^{2}L^{\infty}$-norm of $j$.
\begin{lemma}\label{L34}
Let $(u_{0},b_{0})$ satisfy the conditions stated in Theorem \ref{Th1}, then it holds for any $t\in[0,\,T]$
\begin{eqnarray}\label{t306}
\int_{0}^{t}{
\|j(\tau)\|_{L^{\infty}}^{2}\,d\tau}\leq C(T,u_{0},b_{0}).
\end{eqnarray}
\end{lemma}
\begin{proof}[{Proof of  Lemma \ref{L34}}]
We rewrite the second equation of (\ref{VMHD}) as
\begin{eqnarray} \partial_{t}j+\mathcal{L}j=\nabla\cdot(b\otimes \omega)-\nabla\cdot
(u\otimes j)+T(\nabla u, \nabla b).\nonumber\end{eqnarray}
We again make use of (\ref{solve}) to deduce
\begin{eqnarray}\label{jequ}j(t)=K(t)\ast j_{0}+\int_{0}^{t}{K(t-\tau)\ast [\nabla\cdot(b\otimes \omega)-\nabla\cdot
(u\otimes j)+T(\nabla u, \nabla b)](\tau)\,d\tau}.\end{eqnarray}
Taking $L^{\infty}$-norm in terms of space variable and using the Young inequality, we conclude
\begin{align}
\|j(t)\|_{L^{\infty}}&\leq \|K(t)\ast j_{0}\|_{L^{\infty}}+\int_{0}^{t}{\|\nabla K(t-\tau)\ast [(b\otimes \omega)-
(u\otimes j)](\tau)\|_{L^{\infty}}\,d\tau}\nonumber\\& \quad+
\int_{0}^{t}{\| K(t-\tau)\ast T(\nabla u, \nabla b)(\tau)\|_{L^{\infty}}\,d\tau}
\nonumber\\&\leq C\|K(t)\|_{L^{2}}\| j_{0}\|_{L^{2}}+C\int_{0}^{t}{\| K(t-\tau)\|_{L^{\infty}} \|T(\nabla u, \nabla b)(\tau)\|_{L^{1}}\,d\tau}
\nonumber\\&\quad+C\int_{0}^{t}{\|\nabla K(t-\tau)\|_{L^{2}}(\|(b\otimes \omega)(\tau)\|_{L^{2}}+\|
(u\otimes j)(\tau)\|_{L^{2}})\,d\tau}.\nonumber
\end{align}
Take $L^{2}$-norm in terms of time variable and use the convolution Young inequality as well as the estimates (\ref{t304})-(\ref{t306}) to show
\begin{align}
\|j(t)\|_{L_{t}^{2}L^{\infty}}&\leq C\|K(t)\|_{L_{t}^{2}L^{2}}\| j_{0}\|_{L^{2}}+C\|K(t)\|_{L_{t}^{1}L^{\infty}}\|T(\nabla u, \nabla b)\|_{L_{t}^{2}L^{1}}\nonumber\\&\quad+C\|\nabla K(t)\|_{L_{t}^{1}L^{2}}(\|b\otimes \omega\|_{L_{t}^{2}L^{2}}+\|u\otimes j\|_{L_{t}^{2}L^{2}})\nonumber\\&\leq
C\|K(t)\|_{L_{t}^{2}L^{2}}\| j_{0}\|_{L^{2}}+C\|K(t)\|_{L_{t}^{1}L^{\infty}}\|\omega\|_{L_{t}^{2}L^{2}}
\|j\|_{L_{t}^{2}L^{2}}\nonumber\\&\quad+C\|\nabla K(t)\|_{L_{t}^{1}L^{2}}(\|b\|_{L_{t}^{2}L^{\infty}} \|\omega\|_{L_{t}^{\infty}L^{2}}+\|u\|_{L_{t}^{4}L^{4}} \|j\|_{L_{t}^{4}L^{4}})
\nonumber\\&\leq C(T,u_{0},b_{0}),\nonumber
\end{align}
where we have applied the following facts:
\begin{eqnarray}
\int_{0}^{t}{
\|K(\tau)\|_{L^{2}}^{2}\,d\tau}\leq C(T,u_{0},b_{0}),\qquad\int_{0}^{t}{
\|K(\tau)\|_{L^{\infty}}\,d\tau}\leq C(T,u_{0},b_{0}),\nonumber
\end{eqnarray}
$$\|u\|_{L_{t}^{4}L^{4}} \leq C\|u\|_{L_{t}^{\infty}L^{2}}^{\frac{1}{2}}\|\omega\|_{L_{t}^{2}L^{2}}^{\frac{1}{2}}\leq C(T,u_{0},b_{0})$$
and by (\ref{adduse})
\begin{align}
\|j\|_{L_{t}^{4}L^{4}}&\leq  C\|j\|_{L_{t}^{\infty}L^{2}}^{\frac{1}{2}}\|\nabla j\|_{L_{t}^{2}L^{2}}^{\frac{1}{2}}\nonumber\\&\leq C\|j\|_{L_{t}^{\infty}L^{2}}^{\frac{1}{2}}(\|j\|_{L_{t}^{2}L^{2}}^{\frac{1}{2}}
+\|\mathcal{L}^{\frac{1}{2}}j\|_{L_{t}^{2}L^{2}}^{\frac{1}{2}})
\nonumber\\&\leq C(T,u_{0},b_{0}).\nonumber
\end{align}
This completes the proof of Lemma \ref{L34}.
\end{proof}
\vskip .1in
The next lemma plays a significant role in obtaining the global bound for $(u,\,b)$ in $H^{s}$ with $s>2$ for the system (\ref{GMHD}).
\begin{lemma}\label{L35}
Let $(u_{0},b_{0})$ satisfy the conditions stated in Theorem \ref{Th1}, then it holds for any $t\in[0,\,T]$
\begin{eqnarray}\label{t311}
\int_{0}^{t}{
\|\nabla j(\tau)\|_{L^{\infty}} \,d\tau}\leq C(T,u_{0},b_{0}),
\end{eqnarray}
\begin{eqnarray}\label{addt311}\|\omega(t)\|_{L^{\infty}}\leq C(T,u_{0},b_{0}),\end{eqnarray}
\begin{eqnarray}\label{adddt311}\int_{0}^{t}{
\|\nabla b(\tau)\|_{L^{\infty}}^{2} \,d\tau}\leq C(T,u_{0},b_{0}).\end{eqnarray}
\end{lemma}
\begin{proof}[{Proof of  Lemma \ref{L35}}]
Coming back to (\ref{jequ}), we get
\begin{eqnarray} j(t)=K(t)\ast j_{0}+\int_{0}^{t}{K(t-\tau)\ast [\nabla\cdot(b\otimes \omega)-\nabla\cdot
(u\otimes j)+T(\nabla u, \nabla b)](\tau)\,d\tau}.\nonumber\end{eqnarray}
Applying the gradient $\nabla$ to the both sides of the above equality, taking $L^{\infty}$-norm in terms of space variable and using the Young inequality, it follows that
\begin{align}
\|\nabla j(t)\|_{L^{\infty}}&\leq \|\nabla K(t)\ast j_{0}\|_{L^{\infty}}+\int_{0}^{t}{\|\nabla^{2} K(t-\tau)\ast [(b\otimes \omega)-
(u\otimes j)](\tau)\|_{L^{\infty}}\,d\tau}\nonumber\\&\quad+
\int_{0}^{t}{\|\nabla K(t-\tau)\ast T(\nabla u, \nabla b)(\tau)\|_{L^{\infty}}\,d\tau}
\nonumber\\&\leq C\|\nabla K(t)\|_{L^{2}}\| j_{0}\|_{L^{2}}+C\int_{0}^{t}{\| \nabla K(t-\tau)\|_{L^{2}} \|T(\nabla u, \nabla b)(\tau)\|_{L^{2}}\,d\tau}
\nonumber\\&\quad+C\int_{0}^{t}{\|\nabla^{2} K(t-\tau)\|_{L^{1}}(\|(b\otimes \omega)(\tau)\|_{L^{\infty}}+\|
(u\otimes j)(\tau)\|_{L^{\infty}})\,d\tau}.\nonumber
\end{align}
Now we obtain by taking $L^{1}$-norm in terms of time variable and appealing to the convolution Young inequality
\begin{align}\label{t312} \|\nabla j(t)\|_{L_{t}^{1}L^{\infty}} \leq&C\|\nabla K(t)\|_{L_{t}^{1}L^{2}}\| j_{0}\|_{L^{2}}+C\|\nabla K(t)\|_{L_{t}^{1}L^{2}}\|T(\nabla u, \nabla b)\|_{L_{t}^{1}L^{2}}\nonumber\\ &+C\|\nabla^{2} K(t)\|_{L_{t}^{1}L^{1}}(\|b\otimes \omega\|_{L_{t}^{1}L^{\infty}}+\|u\otimes j\|_{L_{t}^{1}L^{\infty}})\nonumber\\ \leq&
C\|\nabla K(t)\|_{L_{t}^{1}L^{2}}\| j_{0}\|_{L^{2}}+C\|\nabla K(t)\|_{L_{t}^{1}L^{2}}\|\omega\|_{L_{t}^{2}L^{4}}\|j\|_{L_{t}^{2}L^{4}}
\nonumber\\ &+C\|\nabla^{2} K(t)\|_{L_{t}^{1}L^{1}}(\|b\omega\|_{L_{t}^{1}L^{\infty}}
+\|u\|_{L_{t}^{2}L^{\infty}}\|j\|_{L_{t}^{2}L^{\infty}})
\nonumber\\ \leq&
C\|\nabla K(t)\|_{L_{t}^{1}L^{2}}\| j_{0}\|_{L^{2}}\nonumber\\ &+C\|\nabla K(t)\|_{L_{t}^{1}L^{2}}\|\omega\|_{L_{t}^{\infty}L^{2}}^{\frac{1}{2}}
\|\omega\|_{L_{t}^{1}L^{\infty}}^{\frac{1}{2}}
\|j\|_{L_{t}^{\infty}L^{2}}^{\frac{1}{2}}
 \|j\|_{L_{t}^{1}L^{\infty}}^{\frac{1}{2}}
\nonumber\\ &+C\|\nabla^{2} K(t)\|_{L_{t}^{1}L^{1}}(\|b\|_{L_{t}^{\infty}L^{\infty}}
\|\omega\|_{L_{t}^{1}L^{\infty}}+
\|u\|_{L_{t}^{\infty}L^{2}}^{\frac{1}{2}}
\|\omega\|_{L_{t}^{1}L^{\infty}}^{\frac{1}{2}}\| j\|_{L_{t}^{2}L^{\infty}})
\nonumber\\ =&
H_{1}(t)+H_{2}(t)\|\omega\|_{L_{t}^{1}L^{\infty}}^{\frac{1}{2}}
+H_{3}(t)\|\omega\|_{L_{t}^{1}L^{\infty}},
\end{align}
where $H_{l}(t)$ ($l=1,2,3$) are given by
\begin{align}
H_{1}(t)=C\|\nabla K(t)\|_{L_{t}^{1}L^{2}}\| j_{0}\|_{L^{2}};\qquad H_{3}(t)= C\|b\|_{L_{t}^{\infty}L^{\infty}}\|\nabla^{2} K(t)\|_{L_{t}^{1}L^{1}};\nonumber\end{align}
\begin{align}
H_{2}(t)= C(\|\nabla K(t)\|_{L_{t}^{1}L^{2}}\|\omega\|_{L_{t}^{\infty}L^{2}}^{\frac{1}{2}}
\|j\|_{L_{t}^{\infty}L^{2}}^{\frac{1}{2}}
\|j\|_{L_{t}^{1}L^{\infty}}^{\frac{1}{2}}+
\|\nabla^{2} K(t)\|_{L_{t}^{1}L^{1}}
\|u\|_{L_{t}^{\infty}L^{2}}^{\frac{1}{2}}\| j\|_{L_{t}^{2}L^{\infty}}).\nonumber
\end{align}
Thanks to (\ref{tok208}), (\ref{t303}), (\ref{t304}), (\ref{t305}) and (\ref{t306}), it is easy to show that $H_{l}(t)$ ($l=1,2,3$) are non-decreasing functions satisfying
$$H_{l}(t)\leq C(T,u_{0},b_{0}),\quad \forall\,t\in[0,\,T].$$
Consequently, it thus follows from (\ref{t312}) that
\begin{align}\label{t313}
 \int_{0}^{t}\|\nabla j(\tau)\|_{L^{\infty}}\,d\tau  &\leq
H_{1}(t)+\frac{1}{4}H_{2}(t)+H_{2}(t)\|\omega\|_{L_{t}^{1}L^{\infty}}
+H_{3}(t)\|\omega\|_{L_{t}^{1}L^{\infty}}\nonumber\\&=
H_{1}(t)+\frac{1}{4}H_{2}(t)+(H_{2}(t)+H_{3}(t))\int_{0}^{t}\| \omega(\tau)\|_{L^{\infty}}\,d\tau
\nonumber\\&\leq \zeta_{1}(T)
+\zeta_{2}(T)\int_{0}^{t}\| \omega(\tau)\|_{L^{\infty}}\,d\tau,
\end{align}
where
$$\zeta_{1}(T)=H_{1}(T)+\frac{1}{4}H_{2}(T)<\infty,\qquad \zeta_{2}(T)=H_{2}(T)+H_{3}(T)<\infty.$$
 Multiplying the vorticity $\omega$ equation of
(\ref{VMHD}) by $|\omega|^{p-2}\omega$ and integrating over $\mathbb{R}^{2}$
with respect to variable $x$, it holds
\begin{align}
\frac{1}{p}\frac{d}{dt}\|\omega(t)\|_{L^{p}}^{p}=&\int_{\mathbb{R}^{2}} (b\cdot\nabla j) \omega|\omega|^{p-2}\,dx\nonumber\\ \leq&
\|b\|_{L^{\infty}}\|\nabla j\|_{L^{p}} \|\omega\|_{L^{p}}^{p-1}.\nonumber
\end{align}
We thus have
$$\frac{d}{dt}\|\omega(t)\|_{L^{p}}\leq \|b\|_{L^{\infty}}\|\nabla j\|_{L^{p}}.$$
Integrating in time, the outcome is
$$\|\omega(t)\|_{L^{p}}\leq \|\omega(0)\|_{L^{p}}+\int_{0}^{t}\|b(\tau)\|_{L^{\infty}}
\|\nabla j(\tau)\|_{L^{p}}\,d\tau.$$
Letting $p\rightarrow\infty$, it has
\begin{eqnarray} \|\omega(t)\|_{L^{\infty}}\leq \|\omega(0)\|_{L^{\infty}}+\int_{0}^{t}\|b(\tau)\|_{L^{\infty}}
\|\nabla j(\tau)\|_{L^{\infty}}\,d\tau.\nonumber
\end{eqnarray}
By (\ref{t304}), we conclude
\begin{eqnarray}\label{t314}\|\omega(t)\|_{L^{\infty}}\leq \|\omega(0)\|_{L^{\infty}}+C(T)\int_{0}^{t}
\|\nabla j(\tau)\|_{L^{\infty}}\,d\tau.
\end{eqnarray}
Now letting
$$G(t):=\zeta_{1}(T)
+\zeta_{2}(T)\int_{0}^{t}\| \omega(\tau)\|_{L^{\infty}}\,d\tau,$$
it follows from (\ref{t313}) that
$$\int_{0}^{t}\|\nabla j(\tau)\|_{L^{\infty}}\,d\tau\leq G(t).$$
Thanks to (\ref{t314}), we easily get
\begin{align}
\frac{d}{dt}G(t)&=
\zeta_{2}(T)\| \omega(t)\|_{L^{\infty}}\nonumber\\&\leq
\zeta_{2}(T)\big(\|\omega(0)\|_{L^{\infty}}+C(T)\int_{0}^{t}
\|\nabla j(\tau)\|_{L^{\infty}}\,d\tau\big)\nonumber\\&\leq
\zeta_{2}(T)\big(\|\omega(0)\|_{L^{\infty}}+C(T)G(t)\big).
\end{align}
The following key estimate is an easy consequence of the Gronwall lemma
$$G(t)\leq C(T,u_{0},b_{0}),$$
which further implies that
$$\int_{0}^{t}\|\nabla j(\tau)\|_{L^{\infty}}\,d\tau\leq C(T,u_{0},b_{0}).$$
From the above estimate combined with (\ref{t314}), we get
$$\|\omega(t)\|_{L^{\infty}}\leq C(T,u_{0},b_{0}).$$
The following interpolation inequality
$$\|\nabla b\|_{L^{\infty}}\leq C\|\nabla b\|_{L^{2}}^{\frac{1}{2}}\|\nabla j\|_{L^{\infty}}^{\frac{1}{2}},$$
as well as (\ref{t303}) and (\ref{t311}) yields that
$$\int_{0}^{t}{
\|\nabla b(\tau)\|_{L^{\infty}}^{2} \,d\tau}\leq C(T,u_{0},b_{0}).$$
This achieves the proof of Lemma \ref{L35}.
\end{proof}

\vskip .1in
\textbf{The global $H^{s}$ estimate}
To show the global bound for $(u,\,b)$ in $H^{s}$ with $s>2$, we apply $\Upsilon^{s}$ with $\Upsilon:=(I-\Delta)^{\frac{1}{2}}$ to the equations $u$ and $b$, and take the $L^2$ inner product of the resulting equations with $(\Upsilon^{s}u,\,\Upsilon^{s}b)$ to obtain the energy inequality
\begin{align}
& \frac{1}{2}\frac{d}{dt}(\|u(t)\|_{H^{s}}^{2}+\|b(t)\|_{H^{s}}^{2})
+\|\mathcal{L}^{\frac{1}{2}}b\|_{H^{s}}^{2} \nonumber\\
&=  -\int_{\mathbb{R}^{2}}{[\Upsilon^{s}, u\cdot\nabla]u\cdot
\Upsilon^{s}u\,dx}+\int_{\mathbb{R}^{2}}{[\Upsilon^{s},
b\cdot\nabla]b\cdot
\Upsilon^{s}u\,dx}\nonumber\\
& \quad- \int_{\mathbb{R}^{2}}{[\Upsilon^{s}, u\cdot\nabla]b\cdot
\Upsilon^{s}b\,dx}+\int_{\mathbb{R}^{2}}{[\Upsilon^{s},
b\cdot\nabla]u\cdot
\Upsilon^{s}b\,dx}\nonumber\\
&=  J_{1}+J_{2}+J_{3}+J_{4},\nonumber
\end{align}
where $[a,\,b]$ is the standard commutator notation, namely $[a,\,b]=ab-ba$.
By means of the following Kato-Ponce inequality (see \cite{kaPonce})
$$\|[\Upsilon^{s}, f]g\|_{L^{p}}\leq C(\|\nabla f\|_{L^{\infty}}\|\Upsilon^{s-1}g\|_{L^{p}}+\|g\|_{L^{\infty}}\|\Upsilon^{s}f\|_{L^{p}}),\quad 1<p<\infty,$$
one can deduce that
$$J_{1}\leq C\|\nabla u\|_{L^{\infty}}\|u\|_{H^{s}}^{2},\qquad J_{2}\leq C\|\nabla b\|_{L^{\infty}}(\|u\|_{H^{s}}^{2}+\|b\|_{H^{s}}^{2}),$$
$$J_{3}+J_{4}\leq C(\|\nabla u\|_{L^{\infty}}+\|\nabla b\|_{L^{\infty}})(\|u\|_{H^{s}}^{2}+\|b\|_{H^{s}}^{2}).$$
Consequently, it implies that
\begin{eqnarray}
 \frac{d}{dt}(\|u(t)\|_{H^{s}}^{2}+\|b(t)\|_{H^{s}}^{2})
+\|\mathcal{L}^{\frac{1}{2}}b\|_{H^{s}}^{2}
\leq C(\| \nabla u\|_{L^{\infty}}+\|\nabla b\|_{L^{\infty}})(\|{u}\|_{H^{s}}^{2}+\|b\|_{H^{s}}^{2}).\nonumber
\end{eqnarray}
To bound the term $\| \nabla u\|_{L^{\infty}}$ with $\|\omega\|_{L^{\infty}}$, we need the following Sobolev extrapolation inequality with logarithmic correction (see e.g., \cite{Brezis1})
$$\|\nabla u\|_{L^{\infty}(\mathbb{R}^{2})} \leq C\Big(1+\|u\|_{L^{2}(\mathbb{R}^{2})}+
\|\omega\|_{L^{\infty}(\mathbb{R}^{2})} \ln(e+\|
u\|_{H^{s}(\mathbb{R}^{2})})\Big),\quad s>2.$$
Consequently, it enables us to get
\begin{align}
& \frac{d}{dt}(\|u(t)\|_{H^{s}}^{2}+\|b(t)\|_{H^{s}}^{2})
+\|\mathcal{L}^{\frac{1}{2}}b\|_{H^{s}}^{2} \nonumber\\
&\leq  C(1+\|\omega\|_{L^{\infty}}+\|\nabla b\|_{L^{\infty}})\ln(e+\|
u\|_{H^{s}}^{2}+\|b\|_{H^{s}}^{2})(\|{u}\|_{H^{s}}^{2}+\|b\|_{H^{s}}^{2}).
\end{align}
Applying the log-Gronwall type inequality  as well as the estimates (\ref{addt311}) and (\ref{adddt311}), we eventually obtain
$$\|u(t)\|_{H^{s}}+\|b(t)\|_{H^{s}}+\int_{0}^{t}\|\mathcal{L}^{\frac{1}{2}}b(\tau)\|_{H^{s}}^{2}\,d\tau\leq C(t),$$
which together with (\ref{adduse}) implies
$$\int_{0}^{t}\|b(\tau)\|_{H^{s+1}}^{2}\,d\tau\leq C(t).$$
This is nothing but the desired global $H^{s}$ estimates.

\vskip .2in

With the global bounds in the previous lemmas at our disposal, we are ready to
prove Theorem \ref{Th1}.
\begin{proof}[{Proof of Theorem \ref{Th1}}]
With a priori estimates achieved in the previous lemmas, it is a standard procedure to complete
the proof of Theorem \ref{Th1}.
The proof is achieved by using a standard procedure. First we seek the solution
of a regularized system. In order to do this, we recall the mollification
of $\varrho_{N}f$ given by
$$(\varrho_{N}f)(x)=N^{2}\int_{\mathbb{R}^{2}}{\eta\big(N(x-y)\big)f(y)\,dy},$$
where $0\leq\eta(|x|)\in C_{0}^{\infty}(\mathbb{R}^{2})$ satisfies $\int_{\mathbb{R}^{2}}{\eta(y)\,dy}=1$.
Now we regularize our system (\ref{GMHD}) as follows
\begin{equation}
\left\{\aligned &\partial_{t}u^{N}+\mathbb{P}\varrho_{N}((\varrho_{N}u^{N}\cdot\nabla) \varrho_{N}u^{N})=
\mathbb{P}\varrho_{N}((
\varrho_{N}b^{N}\cdot\nabla) \varrho_{N}b^{N}),\\
&\partial_{t}b^{N}+\varrho_{N}((\varrho_{N}u^{N}\cdot\nabla) \varrho_{N}b^{N})+\mathcal{J}\varrho_{N}b^{N}=\varrho_{N}((\varrho_{N}b^{N}\cdot\nabla) \varrho_{N}u^{N}),\\
&\nabla\cdot u^{N}=\nabla\cdot b^{N}=0,\\
&u^{N}(x,0)=\varrho_{N}u_{0}(x),\quad b^{N}(x,0)=\varrho_{N}b_{0}(x),\nonumber
\endaligned \right.
\end{equation}
where $\mathbb{P}$ denotes the Leray projection operator (onto divergence-free vector fields).
For any fixed $N>0$, using properties of mollifiers and following the same argument used in proving the previous lemmas, it is not difficult to establish the
global bound, for any $t\in (0, \infty)$,
$$\|u^{N}(t)\|_{H^{s}}+\|b^{N}(t)\|_{H^{s}}+\int_{0}^{t}\|\mathcal{L}^{\frac{1}{2}}b^{N}(\tau)\|_{H^{s}}^{2}\,d\tau\leq C(t)$$
$$\int_{0}^{t}\|b^{N}(\tau)\|_{H^{s+1}}^{2}\,d\tau\leq C(t).$$
Now the standard Alaoglu's theorem allows us to obtain the global
existence of the classical solution $(u,b)$ to (\ref{GMHD}). The
uniqueness can also be easily established.
This completes the proof of Theorem \ref{Th1}.
\end{proof}

\vskip .2in
\textbf{Acknowledgements.}
The author would
like to thank the anonymous referee and the corresponding editor for their insightful comments and many valuable suggestions, which greatly improved the exposition of the manuscript.
The author was supported by the Foundation of Jiangsu Normal University (No. 16XLR029), the Natural Science Foundation of Jiangsu Province (No. BK20170224), the National Natural Science Foundation of China (No. 11701232).


\vskip .4in
\begin{thebibliography}{00} \frenchspacing
\bibitem{Agelas}
L. Agelas, \emph{Global regularity for logarithmically critical 2D MHD equations with zero viscosity}, Monatsh. Math.  \textbf{181}  (2016), 245-266.

\bibitem{Brezis1}
H. Brezis, S. Wainger, \emph{A note on limiting cases of Sobolev embedding and convolution
inequalities}, Comm. Partial Differential Equations, \textbf{5} (1980), 773-789.

\bibitem{CRegmiW}
 C. Cao, D. Regmi, J. Wu, \emph{The 2D MHD equations
with horizontal dissipation and horizontal magnetic diffusion}, J.
Differential Equations \textbf{254} (2013), 2661-2681.

\bibitem{CW2011}
C. Cao, J. Wu, \emph{Global regularity for the 2D MHD equations with mixed
partial dissipation and magnetic diffusion}, Adv. Math. \textbf{226} (2011),
1803-1822.

\bibitem{CWYSiam14}
C. Cao, J. Wu, B. Yuan, \emph{The 2D incompressible magnetohydrodynamics
equations with only magnetic diffusion}, SIAM J. Math. Anal. \textbf{{46}} (2014), 588-602.

\bibitem{Davidson01}
P.A. Davidson, \emph{An Introduction to Magnetohydrodynamics}, Cambridge University Press, Cambridge, England, 2001.

\bibitem{FNZ14MM}
J. Fan, H. Malaikah, S. Monaquel, G. Nakamura, Y. Zhou, \emph{Global Cauchy problem of 2D generalized MHD equations}, Monatsh. Math. \textbf{{175}} (2014), 127-131.

\bibitem{JZ31114}
Q. Jiu, J. Zhao, \emph{A remark on global regularity of 2D generalized
magnetohydrodynamic equations}, J. Math. Anal. Appl. \textbf{{412}} (2014),
478-484.

\bibitem{JZ31115}
Q. Jiu, J. Zhao, \emph{Global regularity of 2D generalized MHD equations with magnetic diffusion}, Z. Angew. Math. Phys. \textbf{ {66} }(2015), 677-687.

\bibitem{kaPonce}
T. Kato, G. Ponce, \emph{Commutator estimates and the Euler and the Navier-Stokes equations},
Comm. Pure Appl. Math. \textbf{41} (1988), 891-907.

\bibitem{LZ}
Z. Lei, Y. Zhou, \emph{BKM's criterion and global weak solutions for
magnetohydrodynamics with zero viscosity}, Discrete Contin. Dyn.
Syst. \textbf{25} (2009), 575-583.

\bibitem{LinZhang1}
F. Lin, L. Xu, P. Zhang, \emph{Global small solutions of 2-D incompressible MHD system}, J. Differential Equations  \textbf{259}  (2015), 5440-5485.

\bibitem{PF}
E. Priest, T. Forbes, \emph{Magnetic reconnection, MHD theory and
Applications}, Cambridge University Press, Cambridge, 2000.


\bibitem{Renwxz}
X. Ren, J. Wu, Z. Xiang, Z. Zhang, \emph{Global existence and decay of smooth solution for the 2D MHD equations without magnetic diffusion}, J. Funct. Anal. \textbf{267} (2014), 503-541.

\bibitem{ST}
M. Sermange, R. Temam, \emph{Some mathematical questions related to the
MHD equations}, Comm. Pure Appl. Math. \textbf{36} (1983), 635-664.

\bibitem{TYZ}
C. V. Tran, X. Yu, Z. Zhai, \emph{Note on solution regularity of the
generalized magnetohydrodynamic equations with partial dissipation},
Nonlinear Anal. \textbf{85} (2013), 43-51.

\bibitem{TYZ113}
C. V. Tran, X. Yu, Z. Zhai, \emph{On global regularity of 2D generalized
magnetodydrodynamics equations}, J. Differential. Equations \textbf{{254}}
(2013), 4194-4216.

\bibitem{Wu2003}
J. Wu, \emph{The generalized MHD equations}, J. Differential. Equations, \textbf{195} (2003), 284-312.

\bibitem{Wu2011}
J. Wu, \emph{Global regularity for a class of generalized
magnetohydrodynamic equations}, J. Math. Fluid Mech, \textbf{13} (2011), 295-305.

\bibitem{XuZhang}
L. Xu, P. Zhang, \emph{Global small solutions to three-dimensional incompressible MHD system}, SIAM J. Math Anal. \textbf{47} (2015), 26-65.


\bibitem{Y3efg5}
K. Yamazaki, \emph{On the global regularity of two-dimensional generalized
magnetohydrodynamics system}, J. Math. Anal. Appl. \textbf{{416}} (2014), 99-111.

\bibitem{YX2014NA}
Z. Ye, X. Xu, \emph{Global regularity of the two-dimensional
incompressible generalized magnetohydrodynamics system}, Nonlinear
Anal. \textbf{{100}} (2014), 86-96.

\bibitem{YB2014JMAA}
B. Yuan, L. Bai, \emph{Remarks on global regularity of 2D generalized MHD
equations}, J. Math. Anal. Appl. \textbf{{413}} (2014), 633-640.

\bibitem{YZhao16}
B. Yuan, J. Zhao,
\emph{Global regularity of 2D almost resistive MHD Equations},Nonlinear Anal. Real World Appl. \textbf{41} (2018), 53--65.

\bibitem{Zhangt}
T. Zhang, \emph{An elementary proof of the global existence and uniqueness theorem to 2D incompressible non-resistive MHD system}, arXiv:1404.5681v1 [math.AP].
\end{thebibliography}
\end{document}